\documentclass[12pt,oneside]{article}
\usepackage{amsmath,amssymb,amsfonts,amsthm}
\usepackage{color}
\newcommand{\T}{{\cal T}}

\newcommand{\Real}{\mathbb R}

\newcommand{\set}[1]{\left\{#1\right\}}

\def\ov#1{\overline{#1}}

\numberwithin{equation}{section} 
\numberwithin{figure}{section} 

\newtheorem{thm}{Theorem}[section]
\newtheorem{lem}[thm]{Lemma}
\newtheorem{prop}[thm]{Proposition}
\newtheorem{cor}[thm]{Corollary}
\theoremstyle{definition}
\newtheorem{defn}[thm]{Definition}

\theoremstyle{remark}
\newtheorem{example}{Example}
\newtheorem{rem}[thm]{Remark}
\newtheorem*{acknowledgement*}{Acknowledgement}

\newcommand\undersym[2]{\raisebox{-7pt}{\tiny$#2$}{\kern-8pt}\mbox{$#1$}}
\newcommand\undersymm[2]{\raisebox{-8pt}{\tiny$#2$}{\kern-15pt}\mbox{$#1$}}
\newcommand\overast[1]{\raisebox{9pt}{\small$\ast$}{\kern-9pt}\mbox{$#1$}}
\newcommand\overlind[1]{\raisebox{10pt}{\small$\overline{{\hspace{2pt}}\star}$}{\kern-7.5pt}\mbox{$#1$}}
\newcommand\overlinc[1]{\raisebox{10pt}{\tiny$\overline{{\hspace{2pt}}\circ}$}{\kern-7.5pt}\mbox{$#1$}}
\newcommand\overlina[1]{\raisebox{10pt}{\small$\overline{{\hspace{1pt}}\ast}$}{\kern-7.5pt}\mbox{$#1$}}
\newcommand\overcirc[1]{\raisebox{10pt}{\tiny{$\circ$}}{\kern-7.5pt}\mbox{$#1$}}
\newcommand\overdiamond[1]{\raisebox{10pt}{\small$\star$}{\kern-7.5pt}\mbox{$#1$}}

\newcommand\tovercirc[1]{\raisebox{5pt}{\tiny{$\circ$}}{\kern-5.5pt}\mbox{$#1$}}
\newcommand\toverdiamond[1]{\raisebox{5pt}{\tiny$\star$}{\kern-5.5pt}\mbox{$#1$}}
\newcommand\toverast[1]{\raisebox{5pt}{\tiny$\ast$}{\kern-5pt}\mbox{$#1$}}

\begin{document}
\title{\bf Coordinate-free study of Finsler spaces of $H_{p}$-scalar curvature}
\author{{\bf  A. Soleiman and S. G. Elgendi}}
\date{}

\maketitle                     
\vspace{-1.15cm}

\begin{center}
{Department of Mathematics, Faculty of Science,\\ Benha University,   Egypt}
\end{center}
\vspace{-0.8cm}
\begin{center}
E-mails: amr.hassan@fsci.bu.edu.eg, amrsoleiman@yahoo.com\\
{\hspace{1.2cm}}salah.ali@fsci.bu.edu.eg, salahelgendi@yahoo.com
\end{center}
\smallskip

 \maketitle

\bigskip

\begin{center}
\textit{Dedicated to Professor Nabil Youssef on the occasion of his 70th birthday
}\end{center}

\bigskip

\noindent{\bf Abstract.}
  The aim of the present paper is to
provide an \emph{intrinsic} investigation of special  Finsler
spaces of $H_{p}$-scalar curvature and of $H_{p}\,$-constant
curvature. Characterizations of such spaces are shown. Sufficient condition  for Finsler space of $H_{p}$-scalar curvature
to be  of perpendicular scalar curvature is investigated. Necessary and sufficient condition under which a Finsler space of scalar curvature turns into a Finsler space of $H_{p}$-scalar
curvature is given. Further, certain conditions under which a
Finsler manifolds of $H_{p}$-scalar curvature and of scalar
curvature reduce to a Finsler manifold of $H_{p}$-constant
curvature  are obtained. Finally, various  examples  are studied and constructed.


\bigskip

\newpage
\vspace{30truept}\centerline{\Large\bf{Introduction}}\vspace{12pt}
\par

In Riemannian geometry, the study of  the Riemannian manifolds of scalar curvature was very fruitful. It has been contributed  in classifying lots of Riemannian manifolds. For example, the  special manifolds of constant  curvature $-1, 0, +1$. The concept of scalar curvature was extended to Finsler geometry. Most of the special  spaces  are derived from the
fact that the tensor fields (torsions and curvatures)
associated with a linear connection  can be given in
special forms. In Riemannian geometry there exist a unique linear connection, that is, Levi-Civita connection. But in Finsler geometry there are lots of linear connections, for example, Cartan connection , Berwald connection,  ...etc.
 Consequently, the special Finsler spaces are
more numerous than those of Riemannian geometry. Special Finsler
spaces are investigated locally (using local coordinates) by many
authors, see for example (\cite{r14}, \cite{r75}, \cite{yos.2},
\cite{r29}, \cite{r34}, \cite{r42}, \cite{sak.}, \cite{yos.1}). On
the other hand, the global (or intrinsic, free from local
coordinates) investigation of such spaces is very rare in the
literature. Some considerable contributions in this direction are
  \cite{r48}, \cite{sca.} and \cite{r86}.
  \par
In a recent paper \cite{sca.},   some
characterizations of a Finsler space of scalar
curvature are  investigated. Also,   necessary and sufficient conditions
under which a Finsler space of scalar curvature reduces to a
Finsler space of constant curvature are shown.

\par
The present paper is a continuation of ~\cite{sca.}  and \cite{r86}, where we provide an \emph{intrinsic} investigation
of some important special Finsler manifolds related to the Berwald
curvature tensors  namely, Finsler manifold of $H_{p}$-scalar
curvature and of $H_{p}$-constant curvature.

In \cite{yos.1},  Yoshida introduced, locally, the notion of Finsler space of $H_p$-scalar curvature. In this paper, we study, in a coordinate-free fashion, the Finsler spaces of $H_p$-scalar curvature and $H_{p}$-constant curvature. We give a characterization for any Finsler manifold to be of $H_p$-scalar curvature. We find a sufficient  condition under which a Finsler space of $H_p$-scalar curvature is of perpendicular scalar curvature.

Section 3 is devoted to focus on the Finsler spaces of scalar curvature and constant curvature. Necessary and sufficient condition under which a Finsler space of scalar curvature is of $H_p$-scalar curvature is given. We show that every Finsler space of constant curvature is of $H_p$-scalar curvature. The converse is true only in some specific cases, for example, see Theorem 3.6. But, generally, not every Finsler space of $H_p$-constant curvature is of Constant curvature, see Section 4 (Example 6).

In Section 4, we study various examples. Some of these examples are mentioned in the literature, but in different contexts. And some examples are constructed, for instance Example 6.

It should finally be noted that the present work is formulated in a
prospective modern coordinate-free form.
  Moreover,
 the outcome of this work is twofold.
Firstly, the local expressions of the obtained results, when
calculated, coincide with the existing local results. Secondly, new
coordinates-free proofs have been established.

\section{Notations and Preliminaries}

In this section, we give a brief account of the basic concepts
 of the pullback approach to intrinsic Finsler geometry necessary for this work. For more
 details, we refer to \cite{r58, amr3, r48}.

Throughout, $M$ is a smooth manifold of finite dimension $n$. The $\Real$-algebra of smooth real-valued functions on $M$ is denoted by $C^\infty(M)$;  $\mathfrak{X}(M)$ stands for the $C^\infty(M)$-module of vector fields on $M$. The tangent bundle of  $M$ is $\pi_{M}:TM\longrightarrow M$, the subbundle of nonzero tangent vectors
 to $M$ is $\pi: \T M\longrightarrow M$. The vertical subbundle of $TTM$  is denoted by $V(TM)$. The pull-back of  $TM$ over  $\pi$ is $P:\pi^{-1}(TM)\longrightarrow \T M$.

Elements  of  $\mathfrak{X}(\pi (M))$ will be called
$\pi$-vector fields and will be denoted by barred letters
$\overline{X} $. Tensor fields on $\pi^{-1}(TM)$ will be called
$\pi$-tensor fields. The fundamental $\pi$-vector field is the
$\pi$-vector field $\overline{\eta}$ defined by
$\overline{\eta}(u)=(u,u)$ for all $u\in \T M$.

 We have the following short exact sequence of vector bundle morphisms:
\vspace{-0.1cm}
$$0\longrightarrow
 \pi^{-1}(TM)\stackrel{\gamma}\longrightarrow T(\T M)\stackrel{\rho}\longrightarrow
\pi^{-1}(TM)\longrightarrow 0.$$
 Here $\rho := (\pi_{\T M},\pi_\ast)$,  and $\gamma$ is defined by  $\gamma (u,v):=j_{u}(v)$, where
$j_{u}$  is the canonical  isomorphism  from $T_{\pi_{M}(v)}M$ onto $ T_{u}(T_{\pi_{M}(v)}M)$.
Then,   $J:=\gamma\circ\rho$ is a vector $1$-form  on $TM$  called the vertical endomorphism. The Liouville vector field
on $TM$ is the vector field defined by
${C}:=\gamma\circ\overline{\eta},\,\, \overline{\eta}(u)=(u,u),\, u\in TM.$

Let $D$ be  a linear connection (or simply a connection) on the
pullback bundle $\pi^{-1}(TM)$.
 The connection (or the deflection) map associated with $D$ is defined by \vspace{-0.1cm}
$$K:T \T M\longrightarrow \pi^{-1}(TM):X\longmapsto D_X \overline{\eta}
.\vspace{-0.1cm}$$ A tangent vector $X\in T_u (\T M)$ at $u\in \T M$
is horizontal if $K(X)=0$ . The vector space $H_u (\T M)= \{ X \in
T_u (\T M) : K(X)=0 \}$ is called the horizontal space at $u$.
   The connection $D$ is said to be regular if
\begin{equation*}\label{direct sum}
T_u (\T M)=V_u (\T M)\oplus H_u (\T M) \qquad \forall u\in \T M .
\end{equation*}
   Let $\beta:=(\rho |_{H(\T M)})^{-1}$, called the horizontal map of the connection
$D$,  then \vspace{-0.1cm}
   \begin{align*}\label{fh1}
    \rho\circ\beta = id_{\pi^{-1} (TM)}, \quad  \quad
       \beta\circ\rho =   id_{H(\T M)} & {\,\, \text{on}\,\,   H(\T M)}.\vspace{-0.2cm}
\end{align*}

For a regular connection $D$, the horizontal covariant derivative $\stackrel{h}D$ and the vertical covariant
derivatives   $\stackrel{v}D$ are defined, for a
vector (1)$\pi$-form $A$, for example, by \vspace{-0.2cm}
 $$ (\stackrel{h}D A)(\overline{X}, \overline{Y}):=
  (D_{\beta \overline{ X}} A)( \overline{Y}), \quad
  (\stackrel{v}D A)( \overline{X},  \overline{Y}):= (D_{\gamma \overline{X}} A)(  \overline{Y}).$$
\par
 The (classical)  torsion tensor $\textbf{T}$ (resp. the curvature tensor $ \textbf{K}$)  of the connection
$D$ are given by
$$\textbf{T}(X,Y)=D_X \rho Y-D_Y\rho X -\rho [X,Y] ,$$
 $$ \textbf{K}(X,Y)\rho Z=-D_X D_Y \rho Z+D_Y D_X \rho Z+D_{[X,Y]}\rho Z ,$$
 for all $X,Y, Z \in \mathfrak{X} (\T M)$.
 The horizontal
((h)h-) and mixed ((h)hv-) torsion tensors are defined respectively
by \vspace{-0.2cm}
$$Q (\overline{X},\overline{Y}):=\textbf{T}(\beta \overline{X},\beta \overline{Y}),
\, \,\,\,\, T(\overline{X},\overline{Y}):=\textbf{T}(\gamma
\overline{X},\beta \overline{Y}) \quad \forall \,
\overline{X},\overline{Y}\in\mathfrak{X} (\pi (M)).\vspace{-0.2cm}$$
and the horizontal (h-), mixed (hv-) and vertical (v-)
curvature tensors are defined respectively by
$$R(\overline{X},\overline{Y})\overline{Z}:=\textbf{K}(\beta
\overline{X}, \beta \overline{Y})\overline{Z},\quad
P(\overline{X},\overline{Y})\overline{Z}:=\textbf{K}(\beta
\overline{X},\gamma \overline{Y})\overline{Z},\quad
S(\overline{X},\overline{Y})\overline{Z}:=\textbf{K}(\gamma
\overline{X},\gamma \overline{Y})\overline{Z}.$$ The
 (v)h-, (v)hv- and (v)v-torsion tensors  are defined respectively by
$$\widehat{R}(\overline{X},\overline{Y}):={R}(\overline{X},\overline{Y})\overline{\eta},\quad
\widehat{P}(\overline{X},\overline{Y}):={P}(\overline{X},\overline{Y})\overline{\eta},\quad
\widehat{S}(\overline{X},\overline{Y}):={S}(\overline{X},\overline{Y})\overline{\eta}.$$
\par
For a Finsler manifold $(M,L)$, we have  the Berwald connection $D^{\circ} $  on the pullback bundle.

\begin{thm}{\em\cite{r92}} \label{th.1a} Let $(M,L)$ be a Finsler manifold. There exists a
unique regular connection ${{D}}^{\circ}$ on $\pi^{-1}(TM)$ such
that
\begin{description}
 \item[(a)] $D^{\circ}_{h^{\circ}X}L=0$,
  \item[(b)]   ${{D}}^{\circ}$ is torsion-free\,{\em:} ${\textbf{T}}^{\circ}=0 $,
  \item[(c)]The (v)hv-torsion tensor $\widehat{P^{\circ}}$ of ${D}^{\circ}$ vanishes\,\emph{:}
   $\widehat{P^{\circ}}(\overline{X},\overline{Y})= 0$.
  \end{description}
  \par Such a connection is called the Berwald
  connection associated with the Finsler manifold $(M,L)$.
\end{thm}

\section{Finsler Spaces of $H_{p}$-scalar curvature}

In this section, we  investigate (intrinsically)  some important
special Finsler spaces related to the Berwald curvature tensors
namely, Finsler spaces of $H_{p}$-scalar curvature and of $H_{p}$-
constant curvature. Characterizations of such spaces are obtained. Relation between Finsler spaces of $H_{p}$-scalar curvature
and of perpendicular scalar curvature is investigated.

\bigskip
Throughout, for given Finsler manifold $(M,L)$, $g$   denotes
 the Finsler metric on $\pi^{-1}(TM)$  and $\nabla$ denotes the Cartan connection. Also, $T$ stands for the  Cartan tensor, ${R}$ and $\widehat{{R}}$ for   the $h$-curvature and $(v)h$-torsion of Cartan connection,
${\overcirc{R}}$ and $\widehat{\overcirc{R}}$ for   the $h$-curvature and $(v)h$-torsion of Berwald connection, and $H:=i_{\ov \eta}\,\,\widehat{\overcirc{R}}$ for
 the deviation tensor. Moreover,
$\ell:=L^{-1}i_{\overline{\eta}}g$, $\phi(\overline{X}):=\overline{X}-L^{-1}\ell(\ov X)\ov \eta$ and  $\hbar(\overline{X},\ov Y):=g(\phi(\overline{X}),\ov Y)=g(\overline{X},\ov 
Y)-\ell(\overline{X}) \ell(\ov Y)$ the angular metric tensor. Finally,  $\stackrel{h}{D^{\circ}}$  and $\stackrel{v}{D^{\circ}}$  will denote respectively the horizontal covariant derivative and the vertical covariant derivative associated with $D^{\circ}$.

\bigskip

 We begin with the following definitions and results of \cite{sca.} which are useful for subsequent use.
\begin{defn}\label{sca.}{\em{\cite{r86}}} A Finsler manifold $(M,L)$ of dimension $n\geq 3$ is
called  of scalar curvature  $k$  if the deviation tensor $H$
 satisfies
 $$H(\overline{X})=
  k L^{2} \phi(\overline{X}), $$
  where $k(x,y)\in C^\infty(\T M)$ is a positively homogenous of degree zero in $y$ ($x\in M$ and $y\in T_xM$).
   Especially, if the scalar curvature $k$ is constant, then $(M,L)$ is called a Finsler manifold of constant curvature.
\end{defn}

\begin{defn}\label{ind.}\cite{r86} Let $\mathcal{P}$ be the projection operator of indicatrix (or simply, projection operator) and
$\phi(\overline{X}):=\overline{X}-L^{-1}\ell(\overline{X})\ov \eta$. If $\omega$ is a
$\pi$-tensor field  of type {\em(1,p)},
    then $\mathcal{P}\cdot \omega$ is a $\pi$-tensor field  of the same type
    defined by{\,\em:}\vspace{-0.3cm}
\begin{equation*}\label{eq.i2}
    (\mathcal{P}\cdot \omega)(\overline{X}_{1},..., \overline{X}_{p}):= \phi(\omega(\phi(\overline{X}_{1}),..., \phi(\ov
    X_{p}))).
\end{equation*}
   Similarly, if $\omega$ is a $\pi$-tensor field  of type {\em(0,p)},
    then $ {\mathcal{P}}\cdot \omega$ is a $\pi$-tensor field  of the same type
    defined by{\,\em:}\vspace{-0.3cm}
\begin{equation*}\label{eq.i2}
    (\mathcal{P}\cdot \omega)(\overline{X}_{1},..., \overline{X}_{p}):= \omega(\phi(\overline{X}_{1}),..., \phi(\ov
    X_{p})).
    \end{equation*}
    Moreover, for any $\pi$-tensor field $\omega$ is called an indicatory tensor
    if $\mathcal{P}\cdot \omega=\omega$.
\end{defn}

\begin{thm}\cite{sca.}\label{thm.1} A Finsler manifold $(M,L)$ is of scalar curvature $k$ if, and only, if
  the $h$-curvature \,${\overcirc{R}}$ satisfies \footnote{$\mathfrak{A}_{\overline{X},\overline{Y}}\set{\omega(\overline{X},\overline{Y})}:=
\omega(\overline{X},\overline{Y})-\omega(\overline{Y},\overline{X})$}
\begin{eqnarray*}
   {\overcirc{R}}(\overline{X},\ov Y)\ov Z&=& \mathfrak{A}_{\overline{X},\ov Y}\{\phi(\ov Y)\{\ell(\overline{Z})[k\ell(\overline{X})+\frac{1}{3}C^{k}(\overline{X})]
    +\frac{1}{3}B^{k}(\ov Z,\overline{X})\\
   && + \frac{2}{3}\ell(\overline{X})C^{k}(\ov Z)+k\hbar(\ov Z,\overline{X})\}+\frac{1}{3}\ell(\overline{X})C^{k}(\ov Y)\phi(\ov Z)\\
   &&+L^{-1}\hbar(\overline{X},\ov Z)\ov \eta\,
   [k\ell(\ov Y)+\frac{1}{3}C^{k}(\ov Y)]\},
\end{eqnarray*}
where
$C^{k}(\overline{X}):=L(\stackrel{v}{D^{\circ}}k)(\overline{X}),\ \    B^{k}(\ov
X,\ov Y):=L(\mathcal{P}\cdot \stackrel{v}{D^{\circ}}C^{k})(\overline{X}, \ov
Y)$,
and $\stackrel{v}{D^{\circ}}$  is the vertical covariant derivative associated with $D^{\circ}$.
\end{thm}
\begin{thm}\cite{sca.}\label{thm.5} A necessary and sufficient condition for a Finsler manifold of scalar curvature $k$ to be  of constant curvature is that
  the $\pi$-scalar form $C^{k}$ (or $B^{k}$) vanishes.
\end{thm}

\begin{rem} \label{rem.1}
One easily show that the $\pi$-tensor fields $\widehat{P}$, $T$, $\phi$,  $\hbar$, $C^{k}$ and $B^{k}$
are indicatory tensors, and that $\mathcal{P}\cdot \ell$ vanishes identically.
\end{rem}

\bigskip

Now, we are in a position to introduce the definition of  Finsler spaces of $H_{p}$-scalar
curvature.
\begin{defn}\label{Hp-sca.} A Finsler manifold $(M,L)$ of dimension $n\geq 4$ is
called  of $H_{p}$-scalar curvature ($H_{p}$-sc) $\varepsilon$  if the $h$-curvature tensor
 of Berwald connection $\overcirc{R}$  satisfies
 $$(\mathcal{P}\cdot \overcirc{R})(\overline{X}, \ov Y, \ov Z, \ov
 W)=\varepsilon\,\mathfrak{A}_{\overline{X},\ov Y}\,\set{\hbar(\overline{X},\ov Z)\hbar(\ov Y,\ov W)}, $$
  where $\varepsilon(x,y)\in C^\infty(\T M)$ is a positively homogenous of degree zero in $y$.
   Especially, if the $H_{p}$-scalar curvature $\varepsilon$ is constant,
    then $(M,L)$ is called a Finsler manifold of $H_{p}$-constant
    curvature ($H_{p}$-cc).
    Moreover, if the $H_{p}$-scalar curvature $\varepsilon$ vanishes,
    then $(M,L)$ is called a Finsler manifold of vanishing $H_{p}$-scalar curvature.
\end{defn}

Now, we investigate  intrinsically a characterization for a Finsler spaces of $H_{p}$-sc.
\begin{thm} \label{A} A Finsler manifold  is of $H_{p}$-sc $k$, if and only if the h-Berwald curvature \, $\overcirc{R}$ has the form
   \begin{eqnarray}
     \overcirc{R}(\overline{X},\overline{Y},\overline{Z},\overline{w}) &=& L^{-1} \{\ell(\overline{Z}) \widehat{R}(\overline{X},\overline{Y},\overline{W})-
     \ell(\overline{W}) \widehat{R}(\overline{X},\overline{Y},\overline{Z})\} \nonumber \\
     && +\mathfrak{A}_{\overline{X},\overline{Y}}\{
     L^{-1} \ell(\overline{X})\{\widehat{R}(\overline{Z},\overline{W},\overline{Y})-T(H(\overline{Z}),\overline{W},\overline{Y})+
     T(H(\overline{W}),\overline{Z},\overline{Y}) \nonumber \\
     && -T(H(\overline{Y}),\overline{Z},\overline{W})-(D^{o}_{\beta \overline{\eta}} \widehat{P})(\overline{Z},\overline{W},\overline{Y})\} \label{ee}\\
     && -      L^{-2}\{\widehat{R}(\overline{\eta},\overline{X},\overline{Z}) \ell(\overline{Y}) \ell(\overline{W})+
     \widehat{R}(\overline{\eta},\overline{Y},\overline{W}) \ell(\overline{X}) \ell(\overline{Z})-k L^{2}
     \hbar(\overline{X},\overline{Z}) \hbar(\overline{Y},\overline{W})\} \},  \nonumber
   \end{eqnarray}
where  $\widehat{P}$  is the (v)hv-torsion of Cartan connection $\nabla$, $\widehat{P}(\overline{X},\overline{Y},\overline{Z}):=g(\widehat{P}(\overline{X},\overline{Y}),\overline{Z})$ $\widehat{R}(\overline{X},\overline{Y},\overline{Z}):=g(\widehat{R}(\overline{X},\overline{Y}),\overline{Z})$ and  $T(\overline{X},\overline{Y},\overline{Z}):=g(T(\overline{X},\overline{Y}),\overline{Z})$.
\end{thm}

 To prove this theorem, we need the following two lemmas.
 \begin{lem} \label{lem.q}\cite{r96} The h-curvature tensor \, $\overcirc{R}$  of the Berwald connection
has the properties:
\begin{description}
 \item[(a)]
 $\overcirc{R}(\overline{X},\overline{Y},\overline{Z},\overline{W})=- \,
 \overcirc{R}(\overline{Y},\overline{X},\overline{Z},\overline{W})$,

\item[(b)] $  \widehat{\overcirc{R}}(\overline{X},
\overline{Y})= \widehat{R}(\overline{X},
\overline{Y})$,

\item[(c)] $\overcirc{R}(\overline{X},\overline{Y},\overline{Z},\overline{W})+
 \overcirc{R}(\overline{X},\overline{Y},\overline{W},\overline{Z})=2\mathfrak{A}_{\overline{X},\overline{Y}}
 \{(D^\circ_{\beta \overline{Y}}\widehat{P})(\overline{X},\overline{Z},\overline{W})\}-2T(\widehat{R}(\overline{X},\overline{Y}),\overline{Z}, \overline{W})$,

 \item[(d)]  $\mathfrak{S}_{\overline{X},\overline{Y},\overline{Z}}\{
     \overcirc{R}(\overline{X},\overline{Y}) \overline{Z}\}=0$,
 \end{description}
  \end{lem}

 \begin{lem} \label{lem.e} The h-curvature tensor \,$\overcirc{R}$  of the Berwald connection
 satisfies
\begin{eqnarray*}
 \overcirc{R}(\overline{X},\overline{Y},\overline{Z},\overline{W}) &=& \overcirc{R}(\overline{Z},\overline{W},\overline{X},\overline{Y})
 +(D^\circ_{\beta \overline{Y}}\widehat{P})(\overline{X},\overline{Z},\overline{W})-(D^\circ_{\beta \overline{X}}\widehat{P})(\overline{Z},\overline{W},\overline{Y})
  +(D^\circ_{\beta \overline{Z}}\widehat{P})(\overline{X},\overline{W},\overline{Y})\\
   &&-(D^\circ_{\beta \overline{W}}\widehat{P})(\overline{X},\overline{Z},\overline{Y})+T(\widehat{R}(\overline{X},\overline{W}),\overline{Z}, \overline{Y})
   -T(\widehat{R}(\overline{Y},\overline{W}),\overline{X}, \overline{Z})\\
   && +T(\widehat{R}(\overline{Y},\overline{Z}),\overline{X}, \overline{W})-T(\widehat{R}(\overline{X},\overline{Z}),\overline{Y}, \overline{W})
   +T(\widehat{R}(\overline{Z},\overline{W}),\overline{Y}, \overline{X})\\
   && -T(\widehat{R}(\overline{X},\overline{Y}),\overline{Z}, \overline{W}).
\end{eqnarray*}
  \end{lem}
\begin{proof} The proof is a direct consequence of  Lemma \ref{lem.q}.
\end{proof}

\noindent\textbf{Proof of Theorem \ref{A}:} Applying the projection operator  $\mathcal{P}$ on the h-curvature tensor $\overcirc{R}$
taking into account definition \ref{ind.} and Lemma \ref{lem.q}, we obtain
\begin{eqnarray}
  (\mathcal{P} \cdot \overcirc{R})(\overline{X},\overline{Y},\overline{Z},\overline{W})  &=& \overcirc{R} (\overline{Z},\overline{W},\overline{X},\overline{Y})
  -L^{-1}\{\ell(\overline{W}) \,\overcirc{R} (\overline{X}, \overline{Y}, \overline{Z}, \overline{\eta})
  +\ell(\overline{Z}) \, \overcirc{R} (\overline{X}, \overline{Y}, \overline{\eta}, \overline{W})\} \nonumber \\
  &&+\ell(\overline{Y})\, \overcirc{R} (\overline{X}, \overline{\eta}, \overline{Z}, \overline{W})
  +\ell(\overline{X}) \,\overcirc{R} (\overline{\eta}, \overline{Y}, \overline{Z}, \overline{W})\}\nonumber \\
  && +L^{-2} \,
  \overcirc{R}(\overline{X}, \overline{\eta}, \overline{Z}, \overline{\eta}) \ell(\overline{Y}) \ell(\overline{W}) +\overcirc{R}(\overline{X}, \overline{\eta}, \overline{\eta}, \overline{W}) \ell(\overline{Y}) \ell(\overline{Z})
    \label{eq.44}\\
  && +\,\overcirc{R}(\overline{\eta}, \overline{Y}, \overline{Z}, \overline{\eta}) \ell(\overline{X}) \ell(\overline{W})
  +\overcirc{R}(\overline{\eta}, \overline{Y}, \overline{\eta}, \overline{W}) \ell(\overline{X}) \ell(\overline{Z})\} \nonumber.
\end{eqnarray}

\par

On the other hand, from Lemma \ref{lem.q} and Lemma \ref{lem.e}, we obtain
\begin{eqnarray*}
     \overcirc{R}(\overline{\eta}, \overline{Y}, \overline{Z}, \overline{W})&=& \widehat{R}(\overline{Z}, \overline{W}, \overline{Y}) -
    (D^\circ_{\beta \overline{\eta}}\widehat{P})(\overline{Z},\overline{W},\overline{Y})+T(H(\overline{W}),\overline{Z},\overline{Y})\\
    && -T(H(\overline{Z}),\overline{Y},\overline{W})-T(H(\overline{Y}),\overline{Z},\overline{W}).\\
   \overcirc{R}(\overline{\eta}, \overline{Y}, \overline{Z}, \overline{\eta})&=&-\widehat{R}(\overline{\eta}, \overline{Y}, \overline{Z})=g(H(\overline{Y}),\overline{Z}). \\
  \overcirc{R}(\overline{X}, \overline{\eta}, \overline{\eta}, \overline{Y}) &=& -\widehat{R}(\overline{\eta}, \overline{X}, \overline{Y}).
\end{eqnarray*}
By using the above relations,   Equation (\ref{eq.44}) becomes
\begin{eqnarray*}
    (\mathcal{P} \cdot \overcirc{R})(\overline{X},\overline{Y},\overline{Z},\overline{W})  &=& \overcirc{R} (\overline{Z},\overline{W},\overline{X},\overline{Y})- L^{-1} \{\ell(\overline{Z}) \widehat{R}(\overline{X},\overline{Y},\overline{W})-
     \ell(\overline{W}) \widehat{R}(\overline{X},\overline{Y},\overline{Z})\} \nonumber \\
     && -\mathfrak{A}_{\overline{X},\overline{Y}}\{
     L^{-1} \ell(\overline{X})\{\widehat{R}(\overline{Z},\overline{W},\overline{Y})-T(H(\overline{Z}),\overline{W},\overline{Y})+
     T(H(\overline{W}),\overline{Z},\overline{Y}) \nonumber \\
     && -T(H(\overline{Y}),\overline{Z},\overline{W})-(D^{o}_{\beta \overline{\eta}} \widehat{P})(\overline{Z},\overline{W},\overline{Y})\}\\
     && -      L^{-2}\{\widehat{R}(\overline{\eta},\overline{X},\overline{Z}) \ell(\overline{Y}) \ell(\overline{W})+
     \widehat{R}(\overline{\eta},\overline{Y},\overline{W}) \ell(\overline{X}) \ell(\overline{Z})\} \},  \nonumber
   \end{eqnarray*}
Now, if $(M,L)$ is a Finsler manifold of $H_p$-sc $k$, then Equation (\ref{ee}) is satisfied.
\par
Conversely, suppose that $(M,L)$ is a Finsler manifold satisfying Equation (\ref{ee}).
Then, by applying the projection operator  $\mathcal{P}$ on Equation (\ref{ee}) taking Remark \ref{rem.1} into account, one can deduce that
$(M,L)$ is of $H_p$-sc $k$.   $\qquad \qquad\qquad\quad\square$

\bigskip
 In view of Theorem \ref{A} and Remark \ref{rem.1}, we have
\begin{cor} In a Finsler  manifold of $H_p$-sc, we have
\begin{description}
  \item[(a)] $\overcirc{R}(\overline{X},\overline{Y},\overline{Z},\overline{W})+\overcirc{R}(\overline{X},\overline{Y},\overline{W},\overline{Z})
  =2L^{-1} \ell(\overline{Y})\mathfrak{A}_{\overline{X}, \overline{Y}}\{ (D^{o}_{\beta \overline{\eta}} \widehat{P})(\overline{Z},\overline{W},\overline{Y})+
  T(H(\overline{X}),\overline{Z},\overline{W}) \}$
  \item[(b)] $(\mathcal{P} \cdot \overcirc{R})(\overline{X},\overline{Y},\overline{Z},\overline{W})+(\mathcal{P} \cdot \overcirc{R})(\overline{X},\overline{Y},\overline{W},\overline{Z})=0$
\end{description}
\end{cor}

\begin{rem} In Definition \ref{Hp-sca.}, if we replace the h-curvature tensor $\overcirc{R}$ of Berwald connection
by the h-curvature tensor $R$ of Cartan connection, then $(M,L)$ is called of perpendicular scalar curvature  $\varepsilon$.
\end{rem}

Let us define the following tensor:
\begin{equation}\label{Q}
  Q(\overline{X},\overline{Y},\overline{Z},\overline{W}):=g(\widehat{P}(\overline{X},\overline{W}),\widehat{P}(\overline{Y},\overline{Z}))
  -g(\widehat{P}(\overline{X},\overline{Z}),\widehat{P}(\overline{Y},\overline{W})).
\end{equation}

\begin{thm} If $Q(\overline{X},\overline{Y},\overline{Z},\overline{W})=q\set{\hbar(\overline{X}, \overline{Z})\hbar(\overline{Y}, \overline{W})-\hbar( \overline{X}, \overline{W})\hbar( \overline{Y},\overline{Z})}$ holds, then the Finsler manifold of
  $H_p$-sc $k$  is of perpendicular scalar curvature $(k-q)$.
\end{thm}

\begin{proof} By \cite{r96}, we have
$$\overcirc{R}(\overline{X},\overline{Y})\overline{Z}= R(\overline{X},\overline{Y})\overline{Z}- T(\widehat{R}(\overline{X},\overline{Y}),\overline{Z})-
  \mathfrak{A}_{\overline{X},\overline{Y}}\{(\nabla_{\beta \overline{X}}\widehat{P})(\overline{Y}, \overline{Z})
    +\widehat{P}(\overline{X},\widehat{P}(\overline{Y},\overline{Z}))\}.$$
From which together with (\ref{Q}), taking into account the facts that $R(\overline{X},\overline{Y},\overline{Z},\overline{W})=-(\overline{X},\overline{Y},\overline{W},\overline{Z})$
and $g((\nabla_{\beta \overline{X}}\widehat{P})(\overline{Y},\overline{Z}),\overline{W})=g((\nabla_{\beta \overline{X}}\widehat{P})(\overline{Y},\overline{W}),\overline{Z})$, we obtain
\begin{equation*}
  R(\overline{X},\overline{Y},\overline{Z},\overline{W})=\frac{1}{2}\{\overcirc{R} (\overline{X},\overline{Y},\overline{Z},\overline{W})
  -\overcirc{R} (\overline{X},\overline{Y},\overline{W},\overline{Z})\}-Q(\overline{X},\overline{Y},\overline{Z},\overline{W}).
\end{equation*}
Applying the projection operator  $\mathcal{P}$ on both sides of the above relation and using Remark \ref{rem.1}, one can deduce
\begin{equation*}
  (\mathcal{P} \cdot R)(\overline{X},\overline{Y},\overline{Z},\overline{W})=\frac{1}{2}\{(\mathcal{P} \cdot \overcirc{R}) (\overline{X},\overline{Y},\overline{Z},\overline{W})
  -(\mathcal{P} \cdot \overcirc{R}) (\overline{X},\overline{Y},\overline{W},\overline{Z})\}-Q(\overline{X},\overline{Y},\overline{Z},\overline{W}).
\end{equation*}
Now, if $(M,L)$ is a Finsler manifold of $H_p$-s.c. $k$ and using the given assumption for $Q$, we  get
$$(\mathcal{P} \cdot R)(\overline{X},\overline{Y},\overline{Z},\overline{W})=(k-q)\set{\hbar(\overline{X}, \overline{Z})\hbar(\overline{Y}, \overline{W})-\hbar( \overline{X}, \overline{W})\hbar( \overline{Y},\overline{Z})}. $$
This means that the Finsler manifold $(M,L)$ is of perpendicular scalar curvature $(k-q)$.
\end{proof}

\section{  Finsler spaces of scalar curvature and  $H_p$-sc}
Here,  the necessary and sufficient condition under which a
Finsler manifold of scalar curvature turns into a Finsler manifold
of $H_{p}$-scalar curvature is investigated. Moreover, certain condtions
under which a Finsler manifolds of $H_{p}$-scalar curvature and of
scalar curvature reduces to a Finsler manifold of $H_{p}$-constant
curvature or to a Finsler manifold of vanishing $H_{p}$-scalar
curvature are obtained.

\begin{thm}\label{thm.6} Let $(M,L)$ be a Finsler manifold of scalar curvature $k$ of dimension $n\geq
4$. Then, $(M,L)$ is a Finsler manifold of $H_{p}$-scalar curvature
$\varepsilon$ if, and only if, $B^{k}(\overline{X},\ov
Y)=\alpha \hbar(\overline{X},\ov Y)$, where $\alpha=3(\varepsilon -k)$.
\end{thm}

\begin{proof} Suppose that $(M,L)$ is a Finsler manifold of scalar curvature $k$ of dimension $n\geq
4$. Hence, by Theorem \ref{thm.1}, the $h$-curvature tensor
$\overcirc{R}$  has the form
\begin{eqnarray*}
   {\overcirc{R}}(\overline{X},\ov Y,\ov Z, \ov W)&=& \mathfrak{A}_{\overline{X},\ov Y}\{\hbar(\ov Y,\ov W)\{\ell(\ov Z)[k\ell(\overline{X})+\frac{1}{3}C^{k}(\overline{X})]
    +\frac{1}{3}B^{k}(\ov Z,\overline{X})\\
   && + \frac{2}{3}\ell(\overline{X})C^{k}(\ov Z)+k\hbar(\ov Z,\overline{X})\}+\frac{1}{3}\ell(\overline{X})C^{k}(\ov Y)\hbar(\ov Z,\ov W)\\
   &&+L^{-1}\hbar(\overline{X},\ov Z) \ell(\ov W)\,
   [k\ell(\ov Y)+\frac{1}{3}C^{k}(\ov Y)]\}.
\end{eqnarray*}
Applying the projection operator $\mathcal{P}$  on both
sides of the above equation, taking Remark \ref{rem.1} into account, we have
\begin{equation}\label{eq.14a}
    ({\mathcal{P} \cdot \overcirc{R}})(\overline{X},\ov Y,\ov Z, \ov W)=\mathfrak{A}_{\overline{X},\ov
    Y}\{\hbar(\ov Y,\ov W)\{\frac{1}{3}B^{k}(\ov Z,\overline{X})+k\hbar(\ov Z,\ov
    X)\}\}.
\end{equation}
\par
Now, if  $(M,L)$ is a Finsler manifold of $H_{p}$-scalar curvature
$\varepsilon$, then by Definition \ref{Hp-sca.}
\begin{equation}\label{eq.15a}
   (\mathcal{P}\cdot \overcirc{R})(\overline{X}, \ov Y, \ov Z, \ov
 W)=\varepsilon\,\mathfrak{A}_{\overline{X},\ov Y}\,\set{\hbar(\overline{X},\ov Z)\hbar(\ov Y,\ov
 W)}.
\end{equation}
Hence, from (\ref{eq.14a}) and (\ref{eq.15a}), we get
$$\varepsilon\hbar(\ov Z,\overline{X})=\frac{1}{3}B^{k}(\ov Z,\overline{X})+k\hbar(\ov Z,\ov
    X)$$
Consequently,
\begin{equation}\label{eq.16}
     B^{k}(\overline{X},\ov
Y)=3(\varepsilon-k)\hbar(\overline{X},\ov Y)=\alpha\hbar(\overline{X},\ov Y).
\end{equation}
\par
Conversely, suppose that the $\pi$-scalar form $B^{k}$ is given in the form
  (\ref{eq.16}), then,  using Equation (\ref{eq.14a}), one can show that  the $h$-curvature tensor $\overcirc{R}$
has the form
\begin{equation*}
   (\mathcal{P}\cdot \overcirc{R})(\overline{X}, \ov Y, \ov Z, \ov
 W)=\varepsilon\,\mathfrak{A}_{\overline{X},\ov Y}\,\set{\hbar(\overline{X},\ov Z)\hbar(\ov Y,\ov
 W)}.
\end{equation*}
Therefore, from Definition \ref{Hp-sca.}, $(M,L)$ is  a Finsler
manifold of $H_{p}$-scalar curvature.
\end{proof}

\begin{prop}\label{pro.2}A Finsler manifold of constant curvature $k$ of dimension $n\geq
4$ is a Finsler manifold of $H_{p}$-constant curvature
$\varepsilon=k$.
\end{prop}

\begin{proof}If $(M,L)$ is a Finsler manifold of constant curvature $k$ of dimension $n\geq
4$, then the $\pi$-forms $C^{k}$ and $B^{k}$ vanish. Hence, by
Theorem \ref{thm.1}, the $h$-curvature tensor $\overcirc{R}$  has the
form
\begin{eqnarray*}
   {\overcirc{R}}(\overline{X},\ov Y,\ov Z, \ov W)&=& \mathfrak{A}_{\overline{X},\ov Y}\{k\hbar(\ov Y,\ov W)\{\ell(\ov Z)\ell(\overline{X})
   +\hbar(\ov Z,\overline{X})\}+kL^{-1}\hbar(\overline{X},\ov Z) \ell(\ov W)\,
   \ell(\ov Y)\}.
\end{eqnarray*}
Applying the projection operator $\mathcal{P}$  on both
sides of the above Equation, taking  the fact that the
angular metric tensor $\hbar$ is indicator and $\mathcal{P}\cdot \ell=0$ (Remark \ref{rem.1}) into account, we
obtain
\begin{equation*}
   (\mathcal{P}\cdot \overcirc{R})(\overline{X}, \ov Y, \ov Z, \ov
 W)=k\,\mathfrak{A}_{\overline{X},\ov Y}\,\set{\hbar(\overline{X},\ov Z)\hbar(\ov Y,\ov
 W)}.
\end{equation*}
Therefore,  $(M,L)$ is  a Finsler manifold of $H_{p}$-constant
curvature $\varepsilon=k$.
\end{proof}

\begin{rem}\label{rem.2}
The converse of the above proposition dose not hold, in general.
\end{rem}


\begin{prop}\label{pro.3} Let $(M,L)$ be a Finsler manifold of scalar curvature $k$. If $(M,L)$ is of $H_{p}$-scaler curvature
$\varepsilon$, then we have
\begin{description}
  \item[(a)]$3C^{\varepsilon}=2C^{k}$,
  \item[(b)]$3B^{\varepsilon}=2B^{k}$,
  \item[(c)]$3A^{\varepsilon}=2A^{k}$,
\end{description}
where $A^{\varepsilon}:=L\mathcal{P} \cdot
\stackrel{v}{D^{\circ}}B^{\varepsilon}, {\quad} A^{k}:=L\mathcal{P}
\cdot \stackrel{v}{D^{\circ}}B^{k}$.
Moreover, the above assertions are equivalent.
\end{prop}

\begin{proof}~\par

\vspace{4pt}
 \noindent\textbf{(a)}
  Suppose that $(M,L)$ is a Finsler manifold of scalar curvature $k$
  and of $H_{p}$-scaler curvature $\varepsilon$. Then, by Theorem
  \ref{thm.6}, we have
  $$B^{k}(\overline{X},\ov
Y)=3(\varepsilon-k)\hbar(\overline{X},\ov Y).$$ Taking the vertical covariant
derivative on both sides of the above equation, we get
\begin{equation*}
   (\stackrel{v}{D^{\circ}}B^{k})(\ov Z, \overline{X},\ov
Y)=3(\stackrel{v}{D^{\circ}}\varepsilon-\stackrel{v}{D^{\circ}}k)(\ov
Z)\hbar(\overline{X},\ov Y)+3(\varepsilon-k)(\stackrel{v}{D^{\circ}}\hbar)(\ov
Z,\overline{X},\ov Y).
\end{equation*}
Applying the projection operator $\mathcal{P}$  on both
sides of the above equation and multiplying the resulting by $L$, we
obtain
\begin{eqnarray*}
   L(\mathcal{P}\cdot \stackrel{v}{D^{\circ}}B^{k})(\ov Z, \overline{X},\ov
Y)&=&3L(\mathcal{P}\cdot\stackrel{v}{D^{\circ}}\varepsilon-\mathcal{P}\cdot\stackrel{v}{D^{\circ}}k)(\ov
Z)(\mathcal{P}\cdot\hbar)(\overline{X},\ov
Y)\\
&&+6L(\varepsilon-k)\{\mathcal{P}\cdot(T-L^{-1}\hbar\otimes\ell)\}(\ov
Z,\overline{X},\ov Y).
\end{eqnarray*}
In view of  Remark \ref{rem.1}, the above equation reduces to
\begin{eqnarray}\label{eq.18}
   A^{k}(\ov Z, \overline{X},\ov
Y)&=&3\{C^{\varepsilon}(\ov Z)-C^{k}(\ov Z)\}\hbar(\overline{X},\ov
Y)+6L(\varepsilon-k)T(\ov Z,\overline{X},\ov Y).
\end{eqnarray}
On the other hand, using Lemma 3.1 of \cite{sca.}, we have
$$\mathfrak{A}_{\overline{X},\ov Y}\set{A^{k}(\overline{X},\ov Y,\ov Z)+C^{k}(\overline{X})\hbar(\ov Y, \ov Z)}=0,$$
Now, from which together with (\ref{eq.18}), it follows that
$$\mathfrak{A}_{\overline{X},\ov Y}\set{\{3C^{\varepsilon}(\overline{X})-2C^{k}(\overline{X})\}\phi(\ov Y)}=0,$$
 Taking the contracted trace with respect to $\ov Y$, we obtain
\begin{equation*}
(n-2)\{3C^{\varepsilon}(\overline{X})-2C^{k}(\overline{X})\}=0.
\end{equation*}
Hence,
\begin{equation}\label{eq.19}
3C^{\varepsilon}=2C^{k}   \qquad\qquad\qquad(as\,\,    n\geq4).
\end{equation}

\vspace{4pt}
 \noindent\textbf{(b)} Taking the vertical covariant
derivative on both sides of  (\ref{eq.19}), we get
\begin{equation*}
3 \stackrel{v}{D^{\circ}}C^{\varepsilon}=2
\stackrel{v}{D^{\circ}}C^{k}.
\end{equation*}
Applying the projection operator $\mathcal{P}$  on both
sides of the above equation and then multiplying by $L$, we obtain
\begin{equation*}
3B^{\varepsilon}=2B^{k}.
\end{equation*}

\vspace{4pt}
 \noindent\textbf{(c)} The proof can be done in a similar manner as  the proof of \textbf{(b)}.

\bigskip

\textit{\textbf{Now, we prove the equivalence.}}

\vspace{4pt}
 \noindent\textbf{{(a)}\, $\Rightarrow${(b)}\, $\Rightarrow$(c):}
 The proof is obvious.

\vspace{4pt}
 \noindent\textbf{{(c)}\, $\Rightarrow${(a)}:} Suppose that
 $3A^{\varepsilon}=2A^{k}$ holds. Hence, using Lemma 3.1 of \cite{sca.},
 we have
$$\mathfrak{A}_{\overline{X},\ov Y}\set{\{3C^{\varepsilon}(\overline{X})-2C^{k}(\overline{X})\}\hbar(\ov Y,\ov Z)}=0$$
From which, the result follows.
\end{proof}

\begin{lem}\cite{sca.} \label{pro.5} A Finsler manifold of scalar curvature $k$ is of constant curvature if, and only
if, $\mathcal{P}\cdot F=0$, where $F$ is the $\pi$-form defined by
\begin{eqnarray*}
     F(\overline{X},\overline{Y})&:=&\frac{1}{3}\set{B(\overline{X},\overline{Y})+2C(\overline{X})\ell(\overline{Y})} \label{eq.5}.
 \end{eqnarray*}
\end{lem}

\begin{thm}\label{thm.8} Let $(M,L)$ be a Finsler manifold of scalar curvature $k$. If $(M,L)$ is of $H_{p}$-scaler curvature $\varepsilon$, then, the following properties are equivalent:

\begin{description}
  \item[(a)] $\varepsilon=k$.
  \item[(b)]$(M,L)$ is of constant curvature $k$.
  \item[(c)]$(M,L)$ is of $H_{p}$-constant curvature $\varepsilon$.
  \item[(d)]$\mathcal{P}\cdot F=0$
\end{description}
\end{thm}

\begin{proof}Assume that  $(M,L)$ is a Finsler manifold of scalar curvature $k$ and of $H_{p}$-scaler curvature
$\varepsilon$.

\noindent\textbf{(a) $\Longleftrightarrow $(b):}  By making use of  Theorem \ref{thm.6}, we get
\begin{equation}\label{eq.21}
   B^{k}(\overline{X},\ov
Y)=3(\varepsilon-k)\hbar(\overline{X},\ov Y).
\end{equation}
Now, if $(M,L)$ is of constant curvature $k$, then
$$B^{k}(\overline{X},\ov Y)=0$$
From which together with (\ref{eq.21}), it follows that
$\varepsilon=k$.
\par
Conversely, If $(M,L)$ is a Finsler manifold of scalar curvature $k$
and of $H_{p}$-scaler curvature $\varepsilon$ with $\varepsilon=k$.
Then due to (\ref{eq.21}), $B^{k}$ vanishes. Hence, by Theorem
\ref{thm.5}, $(M,L)$ is of constant curvature $k$.

\noindent\textbf{(b)$\Longleftrightarrow $(c):}
Firstly, suppose that $(M,L)$ is a Finsler manifold of scalar curvature $k$ and of $H_{p}\,$-scaler curvature
$\varepsilon$. If $(M,L)$ is of constant curvature $k$, then,  by
Theorem \ref{thm.8}, it follows that $\varepsilon=k$. Hence, $(M,L)$
is of  $H_{p}\,$-constant curvature.
\par
Conversely, If $(M,L)$ is a Finsler manifold  of  $H_{p}\,$-constant
curvature $\varepsilon$. Then, $C^{\varepsilon}$ vanishes. Hence, by
Proposition \ref{pro.3}, $C^{k}$ vanishes. Consequently, by Theorem
\ref{thm.5}, $(M,L)$ is of constant curvature $k$.

\noindent\textbf{(b)$\Longleftrightarrow $(d):} This is a direct consequence of   Lemma  \ref{pro.5}.
\end{proof}

\begin{thm}\label{thm.11} Let $(M,L)$ be a  Finsler manifold of scalar curvature $k$.  Then,  $(M,L)$ is of vanishing $H_{p}$-scaler curvature
 if, and only
if, $\mathcal{P}\cdot N=0$, where $N$ is the $\pi$-form defined
by \begin{eqnarray*}
   N(\overline{X},\overline{Y}):&=&k\set{g(\overline{X},\overline{Y})+\ell(\overline{X})\ell(\overline{Y})}\nonumber\\
   &&+\frac{1}{3}\set{B(\overline{X},\overline{Y})+2\ell(\overline{X})C(\overline{Y})+2C(\overline{X})\ell(\overline{Y})},\label{eq.4}
      \end{eqnarray*}
\end{thm}

\begin{proof}If $(M,L)$ is a Finsler manifold of scalar curvature
$k$. Then, by Theorem \ref{thm.1}, we have
\begin{equation}\label{eq.23}
    ({\mathcal{P} \cdot \overcirc{R}})(\overline{X},\ov Y,\ov Z, \ov W)=\mathfrak{A}_{\overline{X},\ov 
    Y}\{\hbar(\ov Y,\ov W)\{\frac{1}{3}B^{k}(\ov Z,\overline{X})+k\hbar(\ov Z,\ov
    X)\}\}.
\end{equation}
On the other hand, by using the definition of $N$, we obtain
\begin{equation*}
    (\mathcal{P} \cdot N)(\overline{X},\ov Y)=\frac{1}{3}B^{k}(\ov Z,\overline{X})+k\hbar(\ov Z,\ov
    X).
\end{equation*}
From which together with (\ref{eq.23}), we get
\begin{equation}\label{eq.24}
    (\mathcal{P} \cdot \overcirc{R})(\overline{X},\ov Y,\ov Z)=\mathfrak{A}_{\overline{X},\ov
    Y}\{\phi(\ov Y)(\mathcal{P} \cdot N)(\ov Z,\ov
    X)\}.
\end{equation}
Now, if $(M,L)$ is of vanishing $H_{p}$-scalar curvature, then , from
(\ref{eq.24}), it follows that
$$\mathfrak{A}_{\overline{X},\ov Y}\{\phi(\ov Y)(\mathcal{P} \cdot N)(\ov Z,\overline{X})\}.$$
Taking the contracted trace with respect to $\ov Y$, the above
relation reduces to
$$(n-2)(\mathcal{P} \cdot N)(\ov Z,\overline{X})=0.$$
Consequently, as $n\geq3$, $\mathcal{P}\cdot N$ vanishes.
\par
Conversely, If $(M,L)$ is a Finsler manifold  of  scalar curvature
$k$ such that $\mathcal{P}\cdot N=0$. Hence, from (\ref{eq.24}),
$\mathcal{P}\cdot \overcirc{R}$ vanishes. Consequently, $(M,L)$ is of vanishing $H_{p}$-scalar curvature.
\end{proof}


\section{Examples}
In this section, we give various   examples. Although  the definition of the $H_p$-scalaer curvature manifolds requires  that $n\geq 4$, we give special cases of Example 6  with lower dimensions. This only is to show that the converse of some  implications are not, generaly, true.    For instance,  in Example 6, when $n= 2$ we show that the converse of Proposition \ref{pro.2} is not generally true. That is, not every Finsler manifold of $H_p$-constant curvature is of constant curvature. Also,  in dimension 3, we show that not  every Finsler manifold of $H_p$-scalar curvature is of scalar curvature.

\begin{example} \cite{Shen-book} The family of Riemannian  metrics
$${L}_\mu=\frac{\sqrt{|y|^2+\mu(|x|^2|y|^2-\langle x,y\rangle^2)}}{1+\mu|x|^2},\,\, x\in \mathbb{B}^n(r_\mu), \,\, y\in T_x\mathbb{B}^n(r_\mu) \cong\mathbb{R}^n, $$
where $|.|$ and $\langle .,.\rangle$ are the standard Euclidean norm and inner product in $\Real^n$, $r_\mu=1/\sqrt{-\mu}$ if $\mu<0$ and $r_\mu=\infty$ if $\mu\geq 0$. This class is of constant curvature $\mu$. So it is also of $H_p$-constant curvature $\mu$. As a verification, we have the following.
 The spray coefficients $G^i$ of $L_\mu$ are given by
$$G^i=-\frac{\mu\langle x,y \rangle}{1+\mu|x|^2}y^i.$$
From now on, we denote by $\partial_{i}$ and $\dot{\partial}_{i}$  the partial differentiation with respect to the coordinates $x^i$ and $y^i$ respectively and by $(x^i,y^i)$  the induced  coordinates on the  $TM$.

Firstly, let us compute the non linear connection $G^i_j$. Since, $G^i_j=\dot{\partial}_jG^i$, one has
$$G^i_j=-\mu\frac{y^ix^\ell\delta_{\ell j}+\langle x,y\rangle\delta^i_j}{1+\mu|x|^2}.$$
The Berwald connection $G^i_{jk}=\dot{\partial}_kG^i_j$, are given by
$$G^i_{jk}=-\mu\frac{x^\ell\delta_{\ell j}\delta^i_k+x^\ell\delta_{\ell k}\delta^i_j}{1+\mu|x|^2}.$$
To calculate the h-curvature, \,$\overcirc{R}^{\,\,i}_{h\,\,jk}$, of Berwald connection, we have
$$\partial_{h}G^i_{jk}=-\mu\frac{(1+\mu|x|^2)(\delta^i_{k}\delta_{jh}+\delta_{kh}\delta^i_j)-2\mu(x^\ell\delta_{\ell j}\delta^i_k+x^\ell\delta_{\ell k}\delta^i_j)x^m\delta_{m h }}{(1+\mu|x|^2)^2}.$$

Now, $$\overcirc{R}^{\,\,i}_{h\,\,jk}=\mathfrak{A}_{j,k}\{\delta_jG^i_kh+G^i_{mk}G^m_{hj}\}, \quad  \delta_k=\partial_k-N^j_k\dot{\partial}_j,$$ we get
$$\overcirc{R}^{\,\,i}_{h\,\,jk}=\mu(g_{hk}\delta^i_j-g_{hj}\delta^i_k).$$
Applying the projection operator on \, $\overcirc{R}^{\,\,i}_{h\,\,jk}$, we get
$$\mathcal{P} \cdot \overcirc{R}^{\,\,i}_{h\,\,jk}=\overcirc{R}^{\,\,i}_{h\,\,jk}h^h_ah^j_bh^k_ch^d_i=\mu(h_{ac}h^d_b-g_{ab}h^d_c),$$
where $h_{ij}$ is the angular metric. This  gives   $\varepsilon=\mu$.
\end{example}

 \begin{example}\cite{Funk-metric} The Funk metric on the unit ball $\mathbb{B}^n$ in $\Real^n$ is given by
$${L}=\frac{\sqrt{|y|^2-(|x|^2|y|^2-\langle x,y\rangle^2)}}{1-|x|^2}+\frac{\langle x,y\rangle}{1-|x|^2}.$$
The Funk metric has constant curvature $k=-\frac{1}{4}$ and hence has Hp-constant curvature $-\frac{1}{4}$.
\end{example}

\begin{example} \cite{Mu-Elgendi}
The class
$$
L_\phi=\frac{\langle a,y\rangle}{(1+\langle a,x\rangle)^2}\phi(z^1,z^2,...,z^n), \quad z^i=\frac{(1+\langle a,x\rangle)y^i-\langle a,y \rangle x^i}{\langle a,y \rangle},
$$
where $a=(a^1,a^2,...,a^n)$ is a constant vector in $\mathbb{R}^n$ and   $\phi$ is an arbitrary function in the sense  that $L_\phi$ is a Finsler function.
 The spray coefficients are given by $$G^i=-\frac{\langle a,y \rangle}{1+\langle a,x \rangle}y^i. $$
 The class $L_\phi$ is a class of Finsler metric of zero constant curvature and hence it is also of zero $H_{p}$-curvature.
  As in Example 1, straightforward calculations lead to the following.
$$G^i_j=-\frac{y^ia^\ell \delta_{\ell j}+\langle a,y\rangle\delta^i_j}{1+\langle a,x\rangle}, \quad
G^i_{jk}=-\mu\frac{a^\ell \delta_{\ell j}\delta^i_k+a^\ell \delta_{\ell k}\delta^i_j}{1+\langle a,x\rangle},$$
$$\partial_{h}G^i_{jk}=\frac{(a^\ell \delta_{\ell j}\delta^i_k+a^\ell \delta_{\ell k}\delta^i_j)a^m \delta_{m h}}{(1+\langle a,x\rangle)^2}, \quad
\overcirc{R}^{\,\,i}_{h\,\,jk}=0.$$
Consequently, applying the projection operator on \,$ \overcirc{R}^{\,\,i}_{h\,\,jk}$ gives
$$\mathcal{P} \cdot \overcirc{R}^{\,\,i}_{h\,\,jk}=0,$$
which means that $\varepsilon=0$.
\end{example}

As a special case of the above example, we choose $\phi$ as the following example.
\begin{example}
Let $M=\mathbb{R}^n$, and
$$L=\frac{\langle a,y\rangle}{(1+\langle a,x\rangle)^2}\left(|z|^2+e^{-|z|^2}\right)^{\frac{1}{2}}. $$
  $F$ is  a Finsler metric of zero  constant  curvature  and hence zero Hp-constant curvature.
\end{example}

\begin{example}\cite{Shen-example}
Let $M=\mathbb{B}^n(1/\sqrt{|a|})$.
$${L}=\frac{\sqrt{(1-|a|^2|x|^4)|y|^2+(|x|^2\langle a,y\rangle-2\langle a,x\rangle \langle x,y\rangle)^2}-(|x|^2\langle a,y\rangle-2\langle a,x\rangle \langle x,y\rangle)}{1-|a|^2|x|^4}.$$

$L$ is a Finsler metric of scalar curvature $k=3\frac{\langle a,y\rangle}{F}+3\langle a,x\rangle^2-2|a|^2|x|^2$.

Now to decide whether the space $(M,L)$ is of $H_p$-scalar curvature, we have to calculate the tensor $k_{ij}:=F\dot{\partial}_j(F\dot{\partial}_ik)$.
$$\dot{\partial}_ik=3\frac{Fa_i-\langle a,y\rangle \ell_i}{F^2}.$$
Hence,
$$k_{ij}=-3a_i\ell_j-3\frac{\langle a,y\rangle }{F^2}h_{ij}.$$
Applying the projection operator on $k_{ij}$, we have
$$B_{ij}=\mathcal{P}\cdot k_{ij}=k_{ab}h^a_i h^b_j=-3\frac{\langle a,y\rangle }{F}h_{ij}.$$

Therefore, by using Theorem \ref{thm.6}, the space $(M,L)$ is of Hp-scalar curvature. Here the Hp-scalar curvature $\varepsilon$ is given by $$\varepsilon=\frac{2\langle a,y\rangle }{F}+3\langle a,x\rangle^2-2|a|^2|x|^2.$$

\end{example}

\begin{example}

Let $M=\mathbb{R}^n$ and
$$L=f(x^1)|y|,$$
where $f(x^1)$ is an arbitrary smooth positive  function on $\mathbb{R}$ and $|y|$ is the standard norm on $\mathbb{R}^n$.
The metric tensor is given by
$$g_{ij}=f(x^1)^2\delta_{ij},\quad g^{ij}=f(x^1)^{-2}\delta^{ij}, \quad h_{ij}=f(x^1)^2(|y|\delta_{ij}-y_iy_j).$$
The Berwald connection is given by
$$G^h_{ij}=\phi(x^1)(\delta_{1i}\delta^h_j+\delta_{1j}\delta^h_i-\delta_{ij}\delta^h_1), \quad\phi(x^1):=\frac{f'(x^1)}{f(x^1)}, \quad f'(x^1):=\frac{df(x^1)}{dx^1}.$$
By differentiating $G^h_{ij}$  in the following sense $$\partial_k G^h_{ij}=\phi'(x^1)\delta_{1k}(\delta_{1i}\delta^h_j+\delta_{1j}\delta^h_i-\delta_{ij}\delta^h_1),\quad \phi'(x^1):=\frac{d\phi(x^1)}{dx^1},$$
the h-curvature of Berwald connection is given by
$$\overcirc{R}^{\,\,h}_{i\,\,jk}=\mathfrak{A}_{j,k}\{\phi'(\delta_{1i}\delta_{1j}\delta^h_k+\delta_{ij}\delta_{1k}\delta^h_1)+\phi^2(\delta^h_k\delta_{ij}+\delta_{1i}\delta_{1k}\delta^h_j+\delta_{1j}\delta_{ik}\delta^h_1)\}.$$

In dimension $3$, by using Maple program and NF-package \cite{CFG}, one can see that the space $(M,L)$ is not of scalar curvature. But it satisfies the condition of the Hp-scalar curvature. Namely,
$$\mathcal{P}\cdot\overcirc{R}^{\,\,h}_{i\,\,jk}=\overcirc{R}^{\,\,a}_{b\,\,cd}=-\frac{1}{L^2}(\phi'((y^2)^2+(y^3)^2)+\phi (y^1)^2)(h^a_ch_{bd}-h^a_dh_{bc}).$$
Hence, $\varepsilon=-\frac{1}{L^2}(\phi'(((y^2)^2+(y^3)^2)+\phi (y^1)^2)$.

In dimension $2$, by using Maple program and NF-package, one can see that the space $(M,L)$ is of scalar curvature $k=\frac{f'^2-ff''}{f^4}$. But applying the projection operator on the h-curvature gives
$$\mathcal{P}\cdot\overcirc{R}^{\,\,h}_{i\,\,jk}=0.$$
Hence, $\varepsilon=0$. This case shows that the converse of Proposition \ref{pro.2} is not generally true.

\end{example}

\section*{Acknowledgment}
The authors are grateful for Professor Nabil Youssef for his continuous help and encouragement.


\providecommand{\bysame}{\leavevmode\hbox
to3em{\hrulefill}\thinspace}
\providecommand{\MR}{\relax\ifhmode\unskip\space\fi MR }
\providecommand{\MRhref}[2]{%
  \href{http://www.ams.org/mathscinet-getitem?mr=#1}{#2}
} \providecommand{\href}[2]{#2}

\end{document}